\documentclass{article}
\usepackage{graphicx}
\usepackage{amsmath}
\usepackage{amsfonts}
\usepackage[indent=10pt]{parskip}
\usepackage[english]{babel}
\usepackage{babel}
\usepackage{csquotes}
\usepackage{hyperref}
\usepackage{amsthm}
\usepackage{inputenc}
\usepackage{biblatex}
\theoremstyle{definition} 
\newtheorem{theorem}{Theorem}

\newtheorem{proposition}[theorem]{Proposition}
\newtheorem{remark}[theorem]{Remark}
\newtheorem{lemma}[theorem]{Lemma}
\theoremstyle{definition}

\usepackage{biblatex} 
\title{Extreme value theory for geometric Brownian motion and pricing of short maturity  options}
\author{Ng Ze-An}
\date{}

\begin{document}
\maketitle
\begin{abstract}
We investigate the limiting distribution of geometric Brownian motion conditional on its running maximum taking large values. The Freidlin-Wentzell large deviations theory predicts that the conditional distribution of the sample paths converge weakly to a deterministic exponential curve. We complement this result by showing that the conditional sample paths in fact converge in strong sense, and obtain quantitative bounds on the rate of convergence. As an application of our results to financial mathematics, we obtain new closed form asymptotic formulae for the fair price of barrier options with general path dependent payoff in the short maturity limit, with quantitative error estimates. We provide exact formulae for Asian and lookback style payoffs.
\end{abstract}

\section{Introduction}

A natural question in stochastic analysis is the behaviour of stochastic processes, conditional on an event of interest, such as taking a large maximum value. To this end, the Friedlin-Wentzell theory of large deviations provides a powerful framework to analyze such questions. Specialized to stochastic differential equations, the theory implies that conditional on a geometric Brownian motion taking a large maximum value over a fixed timeframe, the process converges weakly, after a suitable normalization to a deterministic exponential curve. 

Equivalently, the same behaviour occurs conditional on hitting a fixed value above its initial value over a short timeframe. As the time frame shrinks, we observe the same concentration along a deterministic exponential curve. We refer to the seminal work \cite{Arnold1984RandomPO} for a comprehensive introduction to the large deviations theory.

There are however, several limitations to the large deviations approach. The primary drawback is that it implies only weak convergence - that is, convergence in law of the process to the limiting deterministic curve. In particular, this means that functionals involving an expectation over sample paths, which form a natural and wide class of functionals, fail to be controlled by the concentration result. Second, it provides no quantitative bound on the convergence rates of the conditional distribution. 

In this work, we complement the Friedlin-Wentzell concentration result in two ways. First, we establish strong convergence of the sample paths - in particlar we show that conditional on hitting a fixed maximum value above its initial value in a short amount of time, the expectation of the supremum norm difference between the normalized paths and the determinsitic curve tends to zero as the timeframe $T$ involved tends to zero. Next, we establish quantitative bounds on the speed of convergence in terms of the timeframe involved in hitting the large maximum - we show that the supremum norm of the difference from the paths to the deterministic curve is at most $O(\sqrt T)$ in expectation as $T$ tends to $0$.

Substantial previous work has been done on the study of strong Friedlin-Wentzell type theorems, An early work in this area is \cite{HULT2005249}, which concerns the conditional strong convergence of pure jump Levy processes conditional on taking a large value. The authors observe a striking phenomenon - as the conditional running maximum grows, the conditional distributions of the Levy process converge to a process with only one jump from the starting value to the running maximum, which occurs at a randomly distributed time. This has been dubbed the "law of one jump" in the literature. These results have since been extended to include processes with weaker regularity, as well as integrals driven by Levy processes.

A significant gap in the current literature on strong Friedlin-Wentzell type convergence theorems is that all studies to date concern pure jump Levy processes - that is, Levy processes with no Brownian component. The case of processes with a continuous component remains a difficult problem. Our paper thus adds to this body of work by establishing an analogous result for the geometric Brownian motion. This complements the existing literature on pure jump processes with a study of Levy processes with pure Brownian component.

As a natural application of our result to financial mathematics, we consider the option pricing problem for barrier options with short maturity. Barrier options are financial instruments that provide a payoff depending on the historical prices of an underlying asset at a future time agreed upon in advance. The barrier component signifies that the option may only be exercised conditional on the price hitting a certain agreed upon price level, otherwise it defaults to a payoff of zero. Due to the sharp discontinuity in the payoff for barrier options, closed form expressions for the fair price of such options are notoriously difficult to obtain and relatively few in the literature.

As a primary consequence of our main results, we obtain closed form asymptotic formulae for the price of barrier options with general path dependent payoff in the short maturity regime, that is, for options priced very close to the maturity time. The payoff is general enough to include the Asian and lookback style payoff, for which we provide explicit formulae. While short maturity limits have been quite extensively studied for Asian options, the case of Barrier options seems to be unexplored to the best of our knowledge. Here, the strong convergence proves critical to obtaining our result, as the price of path-dependent options involves an expectation over sample paths, which weak convergence alone fails to control.

The rest of this paper is organized as follows. We conclude the introduction with a short review of additional related work. In Section 2, we state and prove the main theorem of the paper. In Section 3, we use our results to determine closed form asymptotic expressions for the price of short maturity barrier options with general payoff. We then specialize to Asian, and lookback type payoffs, where more explicit formulae can be provided. We conclude with Section 4, which states some natural further directions for research.

\subsection{Related work}
Strong Freidlin-Wentzell theorems for Levy processes was first studied in \cite{HULT2005249}, in which the authors prove the single jump limiting behaviour for pure jump Levy processes with a suitable notion of regular variation. The results are extended significantly by the same authors in \cite{Hult2007}. In \cite{rhee2017sample}, the authors establish a weak large deviations principle for this scenario, but conclude that a large deviations principle in the classical sense does not hold.

Short maturity options have been quite extensively studied in the Asian case. We list a few references. Pirjol and Zhu \cite{pirjol2016short} (2016) investigates pricing of short maturity Asian options in local volatility models, while Pirjol and Zhu \cite{Pirjolshort} (2017) investigates pricing in the CEV model, a well known stochastic volatility model for stock prices. Meanwhile, \cite{shoshi} uses large deviations theory to study short maturity Asian options in a jump diffusion model.

\section{Main Theorem}

Below we state the main theorem of our paper.
\begin{theorem}[Main Theorem]\label{maintheorem} Let $X$ be the solution to the SDE
$$
d X_{t}=\mu X_{t} d t+\sigma X_{t} d W_{t}, \quad X_{0}=1
$$
with $W$ a standard one dimensional Brownian motion, and $\mu, \sigma>0$ constants. Let $B>1$ be arbitrary. For every $T>0$, let $A_{T}$ denote the event
$$
\left\{\max _{0 \leq t \leq T} X_{t} \geq B\right\}
$$
and let $\mathbb{P}_{T}$ be the probability measure given by
$$
\mathbb{P}_{T}(E)=\frac{\mathbb{P}\left(E \cap A_{T}\right)}{\mathbb{P}\left(A_{T}\right)}
$$
for all events $E$. Denote by $\mathbb{E}_{\mathbb{P}_{T}}$ the expectation under $\mathbb{P}_{T}$. Then we have
$$
\mathbb{E}_{\mathbb{P}_{T}}\left[\sup _{0 \leq t \leq T}\left|X_{t}-B^{\frac{t}{T}}\right|\right]=O(\sqrt{T})
$$
as $T \rightarrow 0^{+}$, where the implied constants in the big $O$ notation depend only on $\mu, \sigma, B$.
\end{theorem}
The majority of our paper will be focused on proving Theorem \ref{maintheorem}. As the proof is lengthy and technical, we seperate the proof of Theorem 1 into a series of four lemmas, followed by the main proof.

\subsection{Technical Lemmas}
First we make some preliminary definitions. For each $M \geq 0$ and $T>0$, denote by $H_{M, T}$ the event $\left\{W_{T} \geq M\right\}$, and let $\mathbb{Q}_{M, T}$ be the probability measure given by
$$
\mathbb{Q}_{M, T}(E)=\frac{\mathbb{P}\left(E \cap H_{M, T}\right)}{\mathbb{P}\left(H_{M, T}\right)}
$$
for all events $E \subset \Omega$. Throughout the first three lemmas, we assume that $f:(0, \infty) \rightarrow \mathbb{R}$ is a function such that $f(x)=O(x)$ as $x \rightarrow 0^{+}$.
\begin{lemma}. We have
$$
\mathbb{E}_{\mathbb{Q}_{M-f(T), T}}\left[\left|W_{T}-M\right|\right]=O(T)
$$
as $T \rightarrow 0^{+}$, where the implied constant in the $O$ notation depends only on $f, M$.
\end{lemma}
\begin{proof}Since $W_{T}$ is a normal random variable with mean 0 and variance $T$, for any bounded nonnegative function $r:[0, \infty) \rightarrow \mathbb{R}$ we have
$$
\begin{aligned}
\mathbb{E}\left[\left|W_{T}-r(T)\right| \mid W_{T} \geq r(T)\right] & =\frac{(2 \pi T)^{-1 / 2} \int_{r(T)}^{\infty} x e^{-\frac{x^{2}}{2 T}} d x}{\mathbb{P}\left(W_{T} \geq r(T)\right)}-r(T) \\
& =\frac{(2 \pi T)^{-1 / 2} \int_{r(T)}^{\infty} x e^{-\frac{x^{2}}{2 T}} d x}{\mathbb{P}\left(Z \geq \frac{r(T)}{\sqrt{T}}\right)}-r(T)
\end{aligned}
$$
where in the second line $Z$ is a standard normal random variable. Writing $\phi$ for the density of the standard normal, noting that $\mathbb{P}(Z \geq x)=\left(1+O\left(\frac{1}{x^{2}}\right)\right) \frac{\phi(x)}{x}$ (see for example, \cite{patel1996handbook}, Chapter 3), we have
$$
\mathbb{E}\left[\left|W_{T}-r(T)\right| \mid W_{T} \geq r(T)\right]=\frac{(2 \pi T)^{-1 / 2} \int_{r(T)}^{\infty} x e^{-\frac{x^{2}}{2 T}} d x}{\left(1+O\left(\frac{T}{r(T)^{2}}\right)\right) \phi\left(\frac{r(T)}{\sqrt{T}}\right) / \frac{r(T)}{\sqrt{T}}}-r(T)
$$

We find by elementary calculus,
$$
\int_{r(T)}^{\infty} x e^{-\frac{x^{2}}{2 T}} d x=T e^{-\frac{r(T)^{2}}{2 T}}
$$

Substituting this into the above, we find
$$
\begin{aligned}
\mathbb{E}\left[\left|W_{T}-r(T)\right| \mid W_{T} \geq r(T)\right] & =\left(\frac{1}{1+O\left(\frac{T}{r(T)^{2}}\right)}\right) r(T)-r(T) \\
& =\left(1+O\left(\frac{T}{r(T)^{2}}\right)\right) r(T)-r(T) \\
& =O\left(\frac{T}{r(T)}\right)
\end{aligned}
$$
as $T \rightarrow 0^{+}$. Setting $r(T)=M-f(T)$, we find that
$$
\mathbb{E}\left[\left|W_{T}-(M-f(T))\right| \mid W_{T} \geq M-f(T)\right]=O(T)
$$
with the implied constant depending only on $M$. Applying the triangle inequality, and recalling that $f(T)$ is of order $O(T)$ then concludes the proof.
\end{proof}
\begin{lemma}For any constant $c>0$, we have
$$
\mathbb{E}_{\mathbb{Q}_{M-f(T), T}}\left[\left|e^{c W_{T}}-e^{c M}\right|\right]=O(\sqrt{T})
$$
as $T \rightarrow 0^{+}$, with the implied constant depending only on $f, c, M$.
\end{lemma}
\begin{proof} Set $\tau_{T}:=\inf \{t>0 \mid \tau \geq M-f(T)\}$. Then we have
$$
\begin{aligned}
\mathbb{E}_{\mathbb{Q}_{M-f(T), T}}\left[e^{c W_{T}}\right] & =\mathbb{E}_{\mathbb{Q}_{M-f(T), T}}\left[e^{c(M-f(T))} e^{c\left(W_{T}-W_{\tau_{T}}\right)}\right] \\
& =e^{c(M-f(T))} \mathbb{E}\left[e^{c\left(W_{T}-W_{\tau_{T}}\right)}\right] \\
& =e^{c(M-f(T))} \exp \left(\frac{c^{2}\left(T-\tau_{T}\right)}{2}\right)
\end{aligned}
$$
which tends to $e^{c M}$ as $T \rightarrow 0^{+}$. In fact, Taylor expanding the exponentials to first order shows that $\mathbb{E}_{\mathbb{Q}_{M-f(T), T}}\left[e^{c W_{T}}\right]$ $e^{c M}$ is of order $O(T)+O(f(T))=O(T)$. Indeed, observe that
$$
\begin{aligned}
\left|\mathbb{E}_{\mathbb{Q}_{M-f(T), T}}\left[e^{c W_{T}}\right]-e^{c M}\right| & =\left\lvert\, e^{c M}\left(\left.1-e^{-f(T)} \exp \left(\frac{c^{2}\left(T-\tau_{T}\right)}{2}\right) \right\rvert\,\right.\right. \\
& =\left|e^{c M}\left(1-(1-f(T)+o(T))\left(1+\frac{c^{2}\left(T-\tau_{T}\right)}{2}+o(T)\right)\right)\right| \\
& =\left|e^{c M}\left(f(T)-\frac{c^{2}\left(T-\tau_{T}\right)}{2}+o(T)\right)\right| \\
& \leq e^{c M}\left(f(T)+\frac{c^{2} T}{2}+o(T)\right) \\
& =O(f(T))+O(T) \\
& =O(T)
\end{aligned}
$$
as claimed. Next, by the Markov inequality we have, for every $\delta>0$,
$$
\mathbb{Q}_{M-f(T), T}\left[\left|W_{T}-M\right| \geq \delta\right] \leq \frac{\mathbb{E}_{\mathbb{Q}_{M-f(T), T}}\left[\left|W_{T}-M\right|\right]}{\delta}
$$

Setting $\delta=\sqrt{T}$, and recalling Lemma 1 , we thus obtain that

\begin{equation}
\mathbb{Q}_{M-f(T), T}\left[\left|W_{T}-M\right| \geq \sqrt{T}\right]=O(\sqrt{T}) 
\end{equation}

Now we compute
$$
\begin{aligned}
& \mathbb{E}_{\mathbb{Q}_{M-f(T), T}}\left[\left|e^{c W_{T}}-e^{c M}\right|\right] \\
& =\mathbb{E}_{\mathbb{Q}_{M-f(T), T}}\left[\mathbf{1}_{\left\{\left|W_{T}-M\right|<\sqrt{T}\right\}}\left|e^{c W_{T}}-e^{c M}\right|\right]+\mathbb{E}_{\mathbb{Q}_{M-f(T), T}}\left[\mathbf{1}_{\left\{\left|W_{T}-M\right| \geq \sqrt{T}\right\}}\left|e^{c W_{T}}-e^{c M}\right|\right] \\
& \leq O(\sqrt{T})+E_{\mathbb{Q}_{M-f(T), T}}\left[\mathbf{1}_{\left\{\left|W_{T}-M\right| \geq \sqrt{T}\right\}}\left|e^{c W_{T}}-e^{c M}\right|\right] .
\end{aligned}
$$

Hence it will suffice to show that the second term above is of order $O(\sqrt{T})$. We write said term as $A_{T}+B_{T}$, where
$$
\begin{aligned}
A_{T} & :=\mathbb{E}_{\mathbb{Q}_{M-f(T), T}}\left[\mathbf{1}_{\left\{W_{T}-M \geq \sqrt{T}\right\}}\left|e^{c W_{T}}-e^{c M}\right|\right] \\
B_{T} & :=\mathbb{E}_{\mathbb{Q}_{M-f(T), T}}\left[\mathbf{1}_{\left\{W_{T}-M \leq-\sqrt{T}\right\}}\left|e^{c W_{T}}-e^{c M}\right|\right] .
\end{aligned}
$$

Observe that $B_{T}=O(\sqrt{T})$. Indeed,
$$
\begin{aligned}
B_{T} & =e^{c M} \mathbb{E}_{\mathbb{Q}_{M-f(T), T}}\left[\mathbf{1}_{\left\{W_{T}-M \leq-\sqrt{T}\right\}}\left|e^{c\left(W_{T}-M\right)}-1\right|\right] . \\
& \leq e^{c M} \mathbb{E}_{\mathbb{Q}_{M-f(T), T}}\left[\mathbf{1}_{\left\{W_{T}-M \leq-\sqrt{T}\right\}}|e+1|\right] . \\
& \leq(e+1) e^{c M} \mathbb{E}_{\mathbb{Q}_{M-f(T), T}}\left[\mathbf{1}_{\left\{\left|W_{T}-M\right| \geq \sqrt{T}\right\}}\right] \\
& =(e+1) e^{c M} O(\sqrt{T}) \\
& =O(\sqrt{T}) .
\end{aligned}
$$
where in the second to last line, we have applied Equation (1). Now we rewrite $A_{T}+B_{T}$ as $A_{T}-B_{T}+2 B_{T}$, and note that
$$
\begin{aligned}
A_{T}-B_{T} & =\mathbb{E}_{\mathbb{Q}_{M-f(T), T}}\left[\mathbf{1}_{\left\{\left|W_{T}-M\right| \geq \sqrt{T}\right\}}\left(e^{c\left(W_{T}-M\right)}-1\right)\right] \\
& =e^{-c M}\left[\left(\mathbb{E}_{\mathbb{Q}_{M-f(T), T}}\left[e^{c W_{T}}\right]-e^{c M}\right)-\mathbb{E}_{\mathbb{Q}_{M-f(T), T}}\left[\mathbf{1}_{\left\{\left|W_{T}-M\right|<\sqrt{T}\right\}}\left(e^{c W_{T}}-e^{c M}\right)\right]\right] .
\end{aligned}
$$

Since the term in brackets is of order $O(T)$ by the earlier discussion, and the latter term is of order $O(\sqrt{T})$, as can be seen by say, Taylor expansion, we obtain that
$$
A_{T}+B_{T}=O(T)+O(\sqrt{T})=O(\sqrt{T})
$$
as desired.
\end{proof}
\begin{lemma}We have
$$
\mathbb{E}_{\mathbb{Q}_{M-f(T), T}}\left[\sup _{0 \leq t \leq T}\left|X_{t}-e^{\frac{t}{T} \sigma M}\right|\right]=O(\sqrt{T})
$$
as $T \rightarrow 0^{+}$, with the implied constant depending only on $f, M$.
\end{lemma}
\begin{proof}Since $X$ is a geometric Brownian motion, it admits the explicit solution
$$
X_{t}=\exp \left(C t+\sigma W_{t}\right)
$$
where for convenience we have written $C:=\mu-\frac{\sigma^{2}}{2}$. Write $W_{t}=\frac{t}{T} W_{T}+B_{t}$, where
$$
B_{t}:=W_{t}-\frac{t}{T} W_{T}
$$
is a standard Brownian bridge, independent of $W_{T}$. We then have
$$
X_{t}=\exp \left(C t-\sigma B_{t}+\frac{\sigma t}{T} W_{T}\right)
$$

Let $D$ be the event $\left\{W_{T} \geq M-f(T)\right\}$. We compute
$$
\begin{aligned}
& \mathbb{E}_{\mathbb{Q}_{M-f(T), T}}\left[\sup _{0 \leq t \leq T}\left|X_{t}-e^{\frac{t}{T} \sigma M}\right|\right] \\
& \leq \mathbb{E}_{\mathbb{Q}_{M-f(T), T}}\left[\sup _{0 \leq t \leq T}\left|\exp \left(C t-\sigma B_{t}+\frac{\sigma t}{T} W_{T}\right)-e^{(\sigma t / T) W_{T}}\right|\right] \\
& +\mathbb{E}_{\mathbb{Q}_{M-f(T), T}}\left[\sup _{0 \leq t \leq T}\left|e^{(\sigma t / T) W_{T}}-e^{\frac{t}{T} \sigma M}\right|\right]
\end{aligned}
$$

By monotonicity, the supremum in the last term occurs at $t=T$, and hence the last term is of order $O(\sqrt{T})$ by Lemma 1. For the first term, we claim that
$$
\begin{aligned}
& \mathbb{E}_{\mathbb{Q}_{M-f(T), T}} {\left[\sup _{0 \leq t \leq T}\left|\exp \left(C t-\sigma B_{t}+\frac{\sigma t}{T} W_{T}\right)-e^{(\sigma t / T) W_{T}}\right|\right] } \\
& \leq \mathbb{E}_{\mathbb{Q}_{M-f(T), T}}\left[\sup _{0 \leq t \leq T}\left|\exp \left(C t-\sigma B_{t}\right)-1\right| e^{\sigma W_{T}}\right] 
\end{aligned}
$$

Indeed, we have trivially
$$
\begin{aligned}
& \left.\sup _{0 \leq t \leq T}\left|\exp \left(C t-\sigma B_{t}+\frac{\sigma t}{T} W_{T}\right)-e^{(\sigma t / T) W_{T}}\right|\right] \\
& \left.\leq \sup _{0 \leq t \leq T} \sup _{0 \leq r \leq T}\left|\exp \left(C t-\sigma B_{t}+\frac{\sigma r}{T} W_{T}\right)-e^{(\sigma r / T) W_{T}}\right|\right] \\
& \left.=\sup _{0 \leq t \leq T} \sup _{0 \leq r \leq T} e^{(\sigma r / T) W_{T}}\left|\exp \left(C t-\sigma B_{t}\right)-1\right|\right]
\end{aligned}
$$

Since $\sigma>0$, and $W_{T}>0$ on the event $D$, we have that for all $t$, the inner supremum is attained at $r=T$, whence Equation 2 follows.

Next, using the independence of $B_{t}$ from $W_{T}$ (for all $t \in 0 \leq t \leq T$ ), denoting by $\mathcal{W}$ the $\sigma$-algebra generated by $W_{T}$, we have
$$
\begin{aligned}
& \mathbb{E}_{\mathbb{Q}_{M-f(T), T}}\left[\sup _{0 \leq t \leq T}\left|\exp \left(C t-\sigma B_{t}+\frac{\sigma t}{T} W_{T}\right)-e^{(\sigma t / T) W_{T}}\right|\right] \\
& \leq \mathbb{E}_{\mathbb{Q}_{M-f(T), T}}\left[\sup _{0 \leq t \leq T} e^{\sigma W_{T}}\left|\exp \left(C t-\sigma B_{t}\right)-1\right|\right] \\
& \leq \frac{\mathbb{E}\left[\mathbf{1}_{D} e^{\sigma W_{T}} \sup _{0 \leq t \leq T}\left|\exp \left(C t-\sigma B_{t}\right)-1\right|\right]}{\mathbb{P}(D)} \\
& =\frac{\mathbb{E}\left[\mathbb{E}\left[\mathbf{1}_{D} e^{\sigma W_{T}} \sup _{0 \leq t \leq T}\left|\exp \left(C t-\sigma B_{t}\right)-1\right| \mid \mathcal{W}\right]\right]}{\mathbb{P}(D)} \\
& =\frac{\mathbb{E}\left[\mathbf{1}_{D} e^{\sigma W_{T}} \mathbb{E}\left[\sup _{0 \leq t \leq T}\left|\exp \left(C t-\sigma B_{t}\right)-1\right| \mid \mathcal{W}\right]\right]}{\mathbb{P}(D)} \\
& =\frac{\mathbb{E}\left[\sup _{0 \leq t \leq T}\left|\exp \left(C t-\sigma B_{t}\right)-1\right|\right] \mathbb{E}\left[1_{D} e^{\left.\sigma W_{T}\right]}\right.}{\mathbb{P}(D)} \\
& =\mathbb{E}\left[\sup _{0 \leq t \leq T}\left|\exp \left(C t-\sigma B_{t}\right)-1\right|\right] \mathbb{E}_{\mathbb{Q}_{M-f(T), T}}\left[e^{\sigma W_{T}}\right]
\end{aligned}
$$

We now claim that $\mathbb{E}\left[\left|\sup _{0 \leq t \leq T} \exp \left(C t-\sigma B_{t}\right)-1\right|\right]$ is of order $O(\sqrt{T})$, while $\mathbb{E}_{\mathbb{Q}_{M-f(T), T}}\left[e^{\sigma W_{T}}\right]$ is of order $O(1)$, whence the result would follow. To see the first claim, note that we have
$$
\begin{aligned}
& \mathbb{E}\left[\left|\sup _{0 \leq t \leq T} \exp \left(C t-\sigma B_{t}\right)-1\right|\right] \\
& \leq \mathbb{E}\left[\mid \sup _{0 \leq t \leq T} \exp \left(C t-\sigma B_{t}\right)-\exp \left(-\sigma B_{t} \mid\right)\right]+\mathbb{E}\left[\left|\sup _{0 \leq t \leq T} \exp \left(-\sigma B_{t}\right)-1\right|\right] \\
& \leq\left(e^{C T}-1\right) \mathbb{E}\left[\mid \sup _{0 \leq t \leq T} \exp \left(-\sigma B_{t}\right)\right]+\mathbb{E}\left[\left|\sup _{0 \leq t \leq T} \exp \left(-\sigma B_{t}\right)-1\right|\right] .
\end{aligned}
$$

Since the former term tends to 0 as $T \rightarrow 0$, it will thus suffice to show that
$$
\mathbb{E}\left[\mid \sup _{0 \leq t \leq T} \exp \left(-\sigma B_{t}\right)-1\right]
$$
tends to 0 . We estimate

\begin{equation}
\mathbb{E}\left[\sup _{0 \leq t \leq T}\left|\exp \left(-\sigma B_{t}\right)-1\right|\right] \leq \mathbb{E}\left[\left|\exp \left(\sup _{0 \leq t \leq T}-\sigma B_{t}\right)-1\right|\right]+\mathbb{E}\left[\left|\exp \left(\inf _{0 \leq t \leq T}-\sigma B_{t}\right)-1\right|\right] .
\end{equation}

We show in turn that both terms in Equation (3) are of order $O(\sqrt{T})$. For the first term, we note that since
$$
B_{t}=W_{t}-\frac{t}{T} W_{T}
$$
we have
$$
0 \leq \sup _{t \in[0, T]}-B_{t} \leq M_{T}+\left|W_{T}\right|
$$
where
$$
M_{t}:=\sup _{t \in[0, T]}-W_{t} .
$$

So by the Cauchy-Schwartz inequality,
$$
\begin{aligned}
\mathbb{E}\left[\mid \sup _{0 \leq t \leq T} \exp \left(-\sigma B_{t}\right)-1\right] & =\mathbb{E}\left[\sup _{0 \leq t \leq T} \exp \left(-\sigma B_{t}\right)-1\right] \\
& \leq \sqrt{\mathbb{E}\left[\exp \left(2 \sigma M_{T}\right)\right]} \sqrt{\mathbb{E}\left[\exp \left(2 \sigma\left|W_{T}\right|\right)\right]}-1
\end{aligned}
$$

By the reflection principle, $M_{T}=\left|W_{T}\right|$ in law, so
$$
\mathbb{E}\left[\sup _{0 \leq t \leq T} \exp \left(-\sigma B_{t}\right)-1\right] \leq \mathbb{E}\left[\exp \left(2 \sigma\left|W_{T}\right|\right)\right]-1
$$

Letting $\Phi$ denote the CDF of a standard normal random variable, by standard formulae, we have
$$
\begin{aligned}
\mathbb{E}\left[\exp \left(2 \sigma\left|W_{T}\right|\right)\right] & =2 e^{2 T \sigma^{2}} \Phi(2 \sigma \sqrt{T}) \\
& =(1+O(\sqrt{T})) e^{2 T \sigma^{2}} \\
& =(1+O(\sqrt{T}))(1+O(T)) \\
& =1+O(\sqrt{T})
\end{aligned}
$$
whence
$$
\mathbb{E}\left[\left|\exp \left(\sup _{0 \leq t \leq T}-\sigma B_{t}\right)-1\right|\right]=O(\sqrt{T})
$$
as claimed.

Now we deal with the second term in Equation 3. Since $\inf _{0 \leq t \leq T}-B_{t} \leq 0$ almost surely, we have $\exp \left(\sigma \inf _{0 \leq t \leq T}-W_{t}\right) \leq 1$, so the second term is
$$
\mathbb{E}\left[1-\exp \left(\sigma \inf _{0 \leq t \leq T}-B_{t}\right)\right]
$$

Hence, it will suffice to show that
$$
\mathbb{E}\left[\exp \left(\sigma \inf _{0 \leq t \leq T}-B_{t}\right)\right] \rightarrow 1
$$

Again, since $B_{t}=W_{t}-\frac{t}{T} W_{T}$, we have
$$
m_{T}-\left|W_{T}\right| \leq \inf _{0 \leq t \leq T}-B_{t} \leq 0
$$
where $m_{T}:=\inf _{0 \leq t \leq T}-W_{t}$. So
$$
\begin{aligned}
E\left[\exp \left(\sigma \inf _{0 \leq t \leq T}-B_{t}\right)\right] & \geq E\left[\exp \left(\sigma\left(m_{T}-\left|W_{T}\right|\right)\right)\right] \\
& =\mathbb{E}\left[\frac{1}{\exp \left(\sigma\left(-m_{T}+\left|W_{T}\right|\right)\right)}\right] \\
& \geq \frac{1}{\mathbb{E}\left[\exp \left(\sigma\left(-m_{T}+\left|W_{T}\right|\right)\right)\right]}
\end{aligned}
$$
where in the last line we have applied Jensen's inequality. Applying the Cauchy Schwartz inequality, we have
$$
\frac{1}{\mathbb{E}\left[\exp \left(\sigma\left(-m_{T}+\left|W_{T}\right|\right)\right)\right]} \geq \frac{1}{\sqrt{\mathbb{E}\left[\exp \left(-2 \sigma m_{T}\right)\right]} \sqrt{\mathbb{E}\left[\exp \left(2 \sigma\left|W_{T}\right|\right)\right]}}
$$

By the reflection principle, $-m_{T}=\left|W_{T}\right|$ in distribution, so
$$
\frac{1}{\sqrt{\mathbb{E}\left[\exp \left(-2 \sigma m_{T}\right)\right]} \sqrt{\mathbb{E}\left[\exp \left(2 \sigma\left|W_{T}\right|\right)\right]}} \geq \frac{1}{\mathbb{E}\left[\exp \left(2 \sigma\left|W_{T}\right|\right)\right]}
$$

Consequently, we have
$$
\mathbb{E}\left[1-\exp \left(\sigma \inf _{0 \leq t \leq T}-B_{t}\right)\right] \leq 1-\frac{1}{\mathbb{E}\left[\exp \left(2 \sigma\left|W_{T}\right|\right)\right]}
$$

Since
$$
\mathbb{E}\left[\exp \left(2 \sigma\left|W_{T}\right|\right)\right]=1+O(\sqrt{T})
$$
as proven earlier, we deduce
$$
\mathbb{E}\left[1-\exp \left(\sigma \inf _{0 \leq t \leq T}-B_{t}\right)\right] \leq 1-\frac{1}{1+O(\sqrt{T})}=O(\sqrt{T})
$$
as claimed.

On the other hand, the second claim follows from a stopping time argument and standard estimates. Indeed, write
$$
\tau_{T}:=\inf \left\{t>0 \mid W_{T} \geq M-f(T)\right\}
$$

We have
$$
\begin{aligned}
\mathbb{E}_{\mathbb{Q}_{M-f(T), T}}\left[e^{\sigma W_{T}}\right] & =\frac{\mathbb{E}\left[1_{D} e^{\sigma W_{T}}\right]}{\mathbb{P}(D)} \\
& =\frac{\mathbb{E}\left[\mathbb{E}\left[\left[1_{D} e^{\sigma W_{T}} \mid \mathcal{F}_{\tau}\right]\right]\right.}{\mathbb{P}(D)} \\
& =\frac{\mathbb{E}\left[1_{D} \mathbb{E}\left[\left[e^{\sigma W_{T}} \mid \mathcal{F}_{\tau}\right]\right]\right.}{\mathbb{P}(D)} .
\end{aligned}
$$

On $D$, we have $\tau_{T} \leq T$ almost surely. Thus by the strong Markov property, conditional on $\mathcal{F}_{\tau}, R_{t}:=W_{t+\tau}$ is a Brownian motion with initial value $R_{0}=W_{\tau}=M-f(T)$. Thus
$$
\begin{aligned}
\mathbb{E}\left[\left[e^{\sigma W_{T}} \mid \mathcal{F}_{\tau}\right]\right. & =\mathbb{E}\left[\left[e^{\sigma R_{t-\tau}} \mid \mathcal{F}_{\tau}\right]\right] \\
& =\left.\mathbb{E}\left[e^{\sigma R_{t-r}}\right]\right|_{r=\tau}
\end{aligned}
$$
where in the last equality we have applied the freezing lemma. We recognise $e^{\sigma R_{t-r}}$ as a log normal random variable with mean $\exp \left(M-f(T)+\frac{t-r}{2}\right) \leq \exp \left(M+|f(T)|+\frac{T}{2}\right)<\exp (M+1):=C$ for all small enough $T$, uniformly over all $0 \leq r \leq t$. Thus
$$
\begin{aligned}
\mathbb{E}_{\mathbb{Q}_{M-f(T), T}}\left[e^{\sigma W_{T}}\right] & =\frac{\mathbb{E}\left[\left.1_{D} \mathbb{E}\left[e^{\sigma R_{t-r}}\right]\right|_{r=\tau}\right]}{\mathbb{P}(D)} \\
& \leq C \frac{\mathbb{E}\left[1_{D}\right]}{\mathbb{P}(D)} \\
& =C
\end{aligned}
$$

Thus $\mathbb{E}_{\mathbb{Q}_{M-f(T), T}}\left[e^{\sigma W_{T}}\right]$ is of order $O(1)$ as claimed, and this concludes the proof.
\end{proof}
In the next lemma, we derive a crucial bound on the hitting time of the geometric Brownian motion $X$ at the given level $B$.

\begin{lemma}[Hitting time bounds] Let $\tau=\inf \left\{t>0 \mid X_{t}=B\right\}$. Then we have
$$
\mathbb{P}\left(\tau \geq\left(1-T^{1 / 2}\right) T \mid \tau \leq T\right) \rightarrow 1
$$
as $T \rightarrow 0^{+}$.
\end{lemma}

\begin{proof} Using that the density $f_{\tau}$ of $\tau$ is given by
$$
f_{\tau}(t)=\frac{R \exp \left(\frac{-\left(R+\left(\frac{\sigma^{2}}{2}-\mu\right) t\right)^{2}}{2 t}\right)}{\sqrt{2 \pi t^{3}}}
$$
we may write
$$
\begin{aligned}
\mathbb{P}\left(\tau \geq\left(1-T^{1 / 2}\right) T \mid \tau \leq T\right) & =\frac{\int_{\left(1-T^{1 / 2}\right) T}^{T} f_{\tau} d t}{\int_{0}^{T} f_{\tau} d t} \\
& =\frac{\int_{\left(1-T^{1 / 2}\right) T}^{T} f_{\tau} d t}{\int_{\left(1-T^{1 / 2}\right) T}^{T} f_{\tau}+\int_{0}^{\left(1-T^{1 / 2}\right) T} f_{\tau} d t} \\
& =: \frac{A_{1}}{A_{1}+A_{2}} \\
& =\frac{1}{1+A_{2} / A_{1}}
\end{aligned}
$$
with $A_{1}:=\int_{\left(1-T^{1 / 2}\right) T}^{T} f_{\tau} d t$ and $A_{2}:=\int_{0}^{\left(1-T^{1 / 2}\right) T} f_{\tau} d t$. Hence it will suffice to show that $\lim _{T \rightarrow 0+} \frac{A_{2}}{A_{1}}=0$. Now we have
$$
\frac{d f_{\tau}}{d t}=\frac{R e^{\left.-\left(2 R+t\left(\sigma^{2}-2 \mu\right)\right)^{2} / 8 t\right)}\left(4 R^{2}-t\left(t\left(2 \mu-\sigma^{2}\right)^{2}\right)+12\right)}{8 \sqrt{2 \pi} t^{7 / 2}}
$$
which is positive on $[0, T]$ for all small enough $T>0$, so $f_{\tau}$ is increasing on this interval.

Thus we may estimate
$$
\begin{aligned}
A_{2} & \leq \int_{0}^{\left(1-T^{1 / 2}\right) T} \frac{R \exp \left(\frac{-\left(R+\left(\frac{\sigma^{2}}{2}-\mu\right)\left(\left(1-T^{1 / 2}\right) T\right)^{2}\right)}{2\left(1-T^{1 / 2}\right) T}\right)}{\sqrt{2 \pi\left(\left(1-T^{1 / 2}\right) T\right)^{3}}} d t \\
& \leq \frac{T R \exp \left(\frac{-\left(R+\left(\frac{\sigma^{2}}{2}-\mu\right)\left(\left(\left(1-T^{1 / 2}\right) T\right)^{2}\right)\right.}{2\left(1-T^{1 / 2}\right) T}\right)}{\sqrt{2 \pi\left(\left(1-T^{1 / 2}\right) T\right)^{3}}}
\end{aligned}
$$
where we have used the fact that $f_{\tau}$ is increasing on $[0, T]$ for small enough $T$. Similarly,
$$
\begin{aligned}
A_{1} & \geq \int_{\left(1-\frac{T^{1 / 2}}{2}\right) T}^{T} f_{\tau} d t \\
& \geq \int_{\left(1-\frac{T^{1 / 2}}{2}\right) T}^{T} \frac{R \exp \left(\frac{-\left(R+\left(\frac{\sigma^{2}}{2}-\mu\right)\left(1-\frac{T^{1 / 2}}{2}\right) T\right)^{2}}{\left.\left(2\left(1-\frac{T^{1 / 2}}{2}\right) T\right)\right)}\right)}{\sqrt{\left.2 \pi\left(2\left(1-\frac{T^{1 / 2}}{2}\right) T\right) / 3\right)^{3}}} d t \\
& =\left(\frac{T^{3 / 2}}{2}\right)\left(\frac{R \exp \left(\frac{-\left(R+\left(\frac{\sigma^{2}}{2}-\mu\right)\left(1-\frac{T^{1 / 2}}{2}\right) T\right)^{2}}{\left.\left(2\left(1-\frac{T^{1 / 2}}{2}\right) T\right)\right)}\right)}{\sqrt{\left.2 \pi\left(2\left(1-\frac{T^{1 / 2}}{2}\right) T\right) / 3\right)^{3}}}\right)
\end{aligned}
$$
so that, after dividing the above two equations we obtain
$$
\frac{A_{2}}{A_{1}} \leq T^{-1 / 2} C_{0} \exp \left(-\frac{C_{1}}{T^{1 / 2}}+C_{2}+C_{3} T\right)
$$
where $C_{0}, \ldots, C_{3}$ are constants with $C_{0}, C_{1}>0$ that do not depend on $T$. We use the simple estimate
$$
\begin{aligned}
\frac{A_{2}}{A_{1}} & \leq T^{-1 / 2} C_{0} \exp \left(-\frac{C_{1}}{T^{1 / 2}}+C_{2}+C_{3}\right) \\
& =C_{4} T^{-1 / 2} \exp \left(-\frac{C_{1}}{T^{1 / 2}}\right)
\end{aligned}
$$
for all $T<1$, say, which tends to 0 as $T \rightarrow 0^{+}$, as desired.
\end{proof}

\subsection{Proof of main theorem}
We are now ready to give the proof of Theorem 1.

\begin{proof}[Proof of Theorem 1] First we show that
\begin{equation}\label{zero}\mathbb{E}_{\mathbb{P}_{T}}\left [\sup _{0<t<T}\left |X_{t}-B^{\frac{t}{T}}\right |\right ] \rightarrow 0\end{equation}
and later refine our analysis to achieve the $O(\sqrt{T})$ convergence rate. To this end, let $Y_{T}$ be the event $\left\{W_{T} \geq\right.$ $\left.G-\left(\frac{\mu}{\sigma}-\frac{\sigma}{2}\right) T\right\}$. We recall that $G:=\frac{\log B}{\sigma}$ and $A_{T}$ is the event $\left\{\max _{0 \leq t \leq T} X_{t} \geq B\right\}$. Note that if $W_{T} \geq G-\left(\frac{\mu}{\sigma} T\right)$, then $X_{T}=\exp \left(\frac{\mu}{\sigma}-\frac{\sigma}{2}\right) T+\sigma W_{T} \geq B$, and thus $Y_{T}$ is a subset of $A_{T}$. We then have

$$\begin{aligned}
& \mathbb{E}_{\mathbb{P}_{T}}\left[\sup _{0<t<T}\left|X_{t}-B^{\frac{t}{T}}\right|\right] \\
& =\mathbb{E}\left[\left.\sup _{0<t<T}\left|X_{t}-B^{\frac{t}{T}}\right| \right\rvert\, A_{T}\right] \\
& =\mathbb{E}\left[\left.\mathbf{1}_{Y_{T}} \sup _{0<t<T}\left|X_{t}-B^{\frac{t}{T}}\right| \right\rvert\, A_{T}\right]+\mathbb{E}\left[\left.\mathbf{1}_{Y_{T}^{c}} \sup _{0<t<T}\left|X_{t}-B^{\frac{t}{T}}\right| \right\rvert\, A_{T}\right] \\
& =\frac{\mathbb{E}\left[\mathbf{1}_{Y_{T}} \mathbf{1}_{A_{T}} \sup _{0<t<T}\left|X_{t}-B^{\frac{t}{T}}\right|\right]}{\mathbb{P}\left(A_{T}\right)}+\frac{\mathbb{E}\left[\mathbf{1}_{Y_{T}^{c}} \mathbf{1}_{A_{T}} \sup _{0<t<T}\left|X_{t}-B^{\frac{t}{T}}\right|\right]}{\mathbb{P}\left(A_{T}\right)} \\
& =\left(\frac{\mathbb{E}\left[\mathbf{1}_{Y_{T}} \mathbf{1}_{A_{T}} \sup _{0<t<T}\left|X_{t}-B^{\frac{t}{T}}\right|\right]}{\mathbb{P}\left(Y_{T}\right)}\right)\left(\frac{\mathbb{P}\left(Y_{T}\right)}{\mathbb{P}\left(A_{T}\right)}\right)+\frac{\mathbb{E}\left[\mathbf{1}_{Y_{T}^{c}} \mathbf{1}_{A_{T}} \sup _{0<t<T}\left|X_{t}-B^{\frac{t}{T}}\right|\right]}{\mathbb{P}\left(A_{T}\right)} \\
& \leq \mathbb{E}\left [\sup _{0<t<T}\left|X_{t}-B^{\frac{t}{T}}\right | \big | \, Y_{T}\right ]+\frac{\mathbb E \left[\mathbf{1}_{Y_{T}^{c}} \mathbf{1}_{A_{T}} \sup_{0<t<T}\left|X_{t}-B^{\frac{t}{T}}\right|\right]}{\mathbb{P}\left(A_{T}\right)} \\
& = \mathbb E_{\mathbb{Q}_{G-(\frac{\mu}{\sigma}-\frac{\sigma}{2})T, T}} \left[\sup _{0<t<T}\left|X_{t}-B^{\frac{t}{T}}\right|\right]+\frac{\mathbb{E}\left[\mathbf{1}_{Y_{T}^{c}} \mathbf{1}_{A_{T}} \sup _{0<t<T}\left|X_{t}-B^{\frac{t}{T}}\right|\right]}{\mathbb{P}\left(A_{T}\right)}
\end{aligned}$$

where in the last two lines we have applied the fact that $Y_{T}$ is a subset of $A_{T}$, hence $\mathbf{1}_{Y_{T}} \mathbf{1}_{A_{T}}=\mathbf{1}_{Y_{T}}$ and $\frac{\mathbb{P}\left(Y_{T}\right)}{\mathbb{P}\left(A_{T}\right)} \leq 1$. We now examine the second term. Writing $\tau:=\inf \left\{t>0 \mid X_{\tau}=B\right\}$, we have

\begin{equation}\label{lol}\begin{aligned}
& \frac{\mathbb{E}\left[\mathbf{1}_{Y_{T}^{c}} \mathbf{1}_{A_{T}} \sup _{0<t<T}\left|X_{t}-B^{\frac{t}{T}}\right|\right]}{\mathbb{P}\left(A_{T}\right)} \\
& \leq \frac{\mathbb{E}\left[\mathbf{1}_{Y_{T}^{c}} \mathbf{1}_{A_{T}} \sup _{0<t<\tau}\left|X_{t}-B^{\frac{t}{T}}\right|\right]}{\mathbb{P}\left(A_{T}\right)}+\frac{\mathbb{E}\left[\mathbf{1}_{Y_{T}^{c}} \mathbf{1}_{A_{T}} \sup _{\tau \leq t<T}\left|X_{t}-B^{\frac{t}{T}}\right|\right]}{\mathbb{P}\left(A_{T}\right)} \\
& =\frac{\mathbb{E}\left[\mathbb{E}\left[\left.\mathbf{1}_{Y_{T}^{c}} \mathbf{1}_{A_{T}} \sup _{0<t<\tau}\left|X_{t}-B^{\frac{t}{T}}\right| \right\rvert\, \mathcal{F}_{\tau}\right]\right]}{\mathbb{P}\left(A_{T}\right)}+\frac{\mathbb{E}\left[\mathbb{E}\left[\left.\mathbf{1}_{Y_{T}^{c}} \mathbf{1}_{A_{T}} \sup _{\tau \leq t<T}\left|X_{t}-B^{\frac{t}{T}}\right| \right\rvert\, \mathcal{F}_{\tau}\right]\right]}{\mathbb{P}\left(A_{T}\right)} \\
& =\frac{\mathbb{E}\left[\mathbb{E}\left[\mathbf{1}_{Y_{T}^{c}} \mid \mathcal{F}_{\tau}\right] \mathbf{1}_{A_{T}} \sup _{0<t<\tau}\left|X_{t}-B^{\frac{t}{T}}\right|\right]}{\mathbb{P}\left(A_{T}\right)}+\frac{\mathbb{E}\left[\mathbb{E}\left[\left.\mathbf{1}_{Y_{T}^{c}} \sup _{\tau \leq t<T}\left|X_{t}-B^{\frac{t}{T}}\right| \right\rvert\, \mathcal{F}_{\tau}\right] \mathbf{1}_{A_{T}}\right]}{\mathbb{P}\left(A_{T}\right)} \\
& \leq \frac{\mathbb{E}\left[\mathbb{E}\left[\mathbf{1}_{Y_{T}^{c}} \mid \mathcal{F}_{\tau}\right] \mathbf{1}_{A_{T}} \sup _{0<t<\tau}\left|X_{t}-B^{\frac{t}{T}}\right|\right]}{\mathbb{P}\left(A_{T}\right)}+\frac{\mathbb{E}\left[\mathbb{E}\left[\left.\sup _{\tau \leq t<T}\left|X_{t}-B^{\frac{t}{T}}\right| \right\rvert\, \mathcal{F}_{\tau}\right] \mathbf{1}_{A_{T}}\right]}{\mathbb{P}\left(A_{T}\right)} \\
& =\frac{\mathbb{E}\left[\mathbb{E}\left[\mathbf{1}_{Y_{T}^{c}} \mid \mathcal{F}_{\tau}\right] \mathbf{1}_{A_{T}} \sup _{0<t<\tau}\left|X_{t}-B^{\frac{t}{T}}\right|\right]}{\mathbb{P}\left(A_{T}\right)}+\mathbb{E}\left[\left.\sup _{\tau \leq t<T}\left|X_{t}-B^{\frac{t}{T}}\right| \right\rvert\, \mathcal{A}_{T}\right]
\end{aligned}\end{equation}

where we have performed an intermediate conditioning on $\mathcal{F}_{\tau}$, and made use of the $\left\{\mathcal{F}_{t}\right\}$-adaptedness of $X$ to bring terms outside the conditional expectation.

We now make two claims - the first is that
$$\mathbb{E}\left[\mathbf{1}_{Y_{T}^{c}} \mid \mathcal{F}_{\tau}\right]$$

is almost surely bounded away from 1 as $T \rightarrow 0$ - that is, there exists some $0<C<1$ and $T_{0}>0$ such that
\begin{equation}\label{claim1}
\mathbb{E}\left[\mathbf{1}_{Y_{T}^{c}} \mid \mathcal{F}_{\tau}\right] \leq C 
\end{equation}
almost surely whenever $T<T_{0}$. The second is that
\begin{equation}\label{claim2}
\mathbb{E}\left[\sup _{\tau<t<T}\left |X_{t}-B^{\frac{t}{T}}\right | \right ] \to 0 \end{equation}
as $T \rightarrow 0^{+}$. Admitting for now these two claims, letting $\varepsilon>0$ be arbitrary, we have
$$\mathbb{E}_{\mathbb{P}_{T}}\left[\sup _{0<t<T}\left|X_{t}-B^{\frac{t}{T}}\right|\right]$$

$$
\begin{aligned}
\leq & \mathbb E_{\mathbb{Q}_{G-(\frac{\mu}{\sigma}-\frac{\sigma}{2})T, T}}\left[\sup _{0<t<T}\left|X_{t}-B^{\frac{t}{T}}\right|\right]+\frac{\mathbb{E}\left[\mathbb{E}\left[\mathbf{1}_{Y_{T}^{c}} \mid \mathcal{F}_{\tau}\right] \mathbf{1}_{A_{T}} \sup _{0<t<\tau}\left|X_{t}-B^{\frac{t}{T}}\right|\right]}{\mathbb{P}\left(A_{T}\right)} \\
& +\mathbb{E}\left[\left.\sup _{\tau<t<T}\left|X_{t}-B^{\frac{t}{T}}\right| \right\rvert\, \mathcal{A}_{T}\right] \\
\leq & \mathbb E_{\mathbb{Q}_{G-(\frac{\mu}{\sigma}-\frac{\sigma}{2})T, T}}\left[\sup _{0<t<T}\left|X_{t}-B^{\frac{t}{T}}\right|\right]+\frac{\mathbb{E}\left[C \mathbf{1}_{A_{T}} \sup _{0<t<T}\left|X_{t}-B^{\frac{t}{T}}\right|\right]}{\mathbb{P}\left(A_{T}\right)} \\
& +\mathbb{E}\left[\left.\sup _{\tau<t<T}\left|X_{t}-B^{\frac{t}{T}}\right| \right\rvert\, \mathcal{A}_{T}\right] \\
= & \mathbb E_{\mathbb{Q}_{G-(\frac{\mu}{\sigma}-\frac{\sigma}{2})T, T}}[\sup_{0<t<T} |X_{t}-B^{\frac{t}{T}}|]+C \mathbb{E}_{\mathbb{P}_{T}} [\sup_{0<t<T} |X_t-B^{\frac{t}{T}}|]+\varepsilon \\
= & C \mathbb{E}_{\mathbb{P}_{T}}\left[\sup _{0<t<T}\left|X_{t}-B^{\frac{t}{T}}\right|\right]+2 \varepsilon
\end{aligned}
$$

for all small enough $T$, where in the third to last line we have applied (\ref{claim1}), in the second to last line we have applied (\ref{claim2}), and in the last line we have applied Lemma 4. Thus
$$
(1-C) \mathbb{E}_{\mathbb{P}_{T}}\left[\sup _{0<t<T}\left|X_{t}-B^{\frac{t}{T}}\right|\right] \leq 2 \varepsilon
$$
which implies
$$
\mathbb{E}_{\mathbb{P}_{T}}\left[\sup _{0<t<T}\left|X_{t}-B^{\frac{t}{T}}\right|\right] \leq \frac{2 \varepsilon}{1-C}
$$

Since $\varepsilon$ was arbitrary, we conclude
$$
\mathbb{E}_{\mathbb{P}_{T}}\left[\sup _{0<t<T}\left|X_{t}-B^{\frac{t}{T}}\right|\right]
$$
tends to 0 as $T \rightarrow 0^{+}$as required.

It remains only to prove the earlier two claims (\ref{claim1}) and (\ref{claim2}). For the first claim, we note that
$$
\mathbb{E}\left[\mathbf{1}_{Y_{T}^{c}} \mid F_{\tau}\right]=1-\mathbb{E}\left[\mathbf{1}_{Y_{T}} \mid F_{\tau}\right]
$$

Hence it will suffice to show that there is some $C>0$ such that $\mathbb{E}\left[\mathbf{1}_{Y_{T}} \mid F_{\tau}\right]>C$ almost surely for all small enough $T$. To this end, we estimate
$$
\begin{aligned}
\mathbb{E}\left[\mathbf{1}_{Y_{T}} \mid F_{\tau}\right] & =\mathbb{P}\left(\left.W_{T} \geq G-\left(\frac{\mu}{\sigma}-\frac{\sigma}{2} T\right) \right\rvert\, \mathcal{F}_{\tau}\right) \\
& =\mathbb{P}\left(\left.W_{\tau}+W_{T}-W \tau \geq G-\left(\frac{\mu}{\sigma}-\frac{\sigma}{2}\right) T \right\rvert\, \mathcal{F}_{\tau}\right) \\
& =\mathbb{P}\left(\left.W_{T}-W \tau \geq\left(\frac{\sigma}{2}-\frac{\mu}{\sigma}\right)(T-\tau) \right\rvert\, \mathcal{F}_{\tau}\right) .
\end{aligned}
$$

Recalling that $W_{T}-W_{\tau}$ is a normal random variable with variance $T-\tau$, we have
$$
\mathbb{P}\left(\left.W_{T}-W \tau \geq\left(\frac{\sigma}{2}-\frac{\mu}{\sigma}\right)(T-\tau) \right\rvert\, \mathcal{F}_{\tau}\right)=\mathbb{P}\left(Z \geq\left(\frac{\sigma}{2}-\frac{\mu}{\sigma}\right) \sqrt{T-\tau}\right)
$$
where $Z$ is a standard normal random variable. The above tends to $\mathbb{P}(Z \geq 0)$ as $T \rightarrow 0$, uniformly in $\omega$, and so any $0<C<\frac{1}{2}$ will satisfy the required inequality, say $C=\frac{1}{3}$. This proves (\ref{claim1}). For the second claim (\ref{claim2}), we estimate, for any $\delta>0$,
$$
\begin{aligned}
& \mathbb{E}\left[\left.\sup _{\tau<t<T}\left|X_{t}-B^{\frac{t}{T}}\right| \right\rvert\, \mathcal{A}_{T}\right] \\
& =\mathbb{E}\left[\left.\mathbf{1}_{\tau<(1-\delta) T} \sup _{\tau<t<T}\left|X_{t}-B^{\frac{t}{T}}\right| \right\rvert\, \mathcal{A}_{T}\right]+\mathbb{E}\left[\left.\mathbf{1}_{\tau \geq(1-\delta) T} \sup _{\tau \leq t<T}\left|X_{t}-B^{\frac{t}{T}}\right| \right\rvert\, \mathcal{A}_{T}\right]
\end{aligned}
$$

The first term above is equal to
$$
\begin{aligned}
& \frac{\mathbb{E}\left[\mathbf{1}_{A_{T}} \mathbf{1}_{\{\tau<(1-\delta) T\}} \sup _{\tau \leq t<T}\left|X_{t}-B^{\frac{t}{T}}\right|\right]}{\mathbb{P}\left(A_{T}\right)} \\
& =\frac{\mathbb{E}\left[\mathbb{E}\left[\mathbf{1}_{A_{T}} \mathbf{1}_{\{\tau<(1-\delta) T\}} \sup _{\tau \leq t<T} \left\lvert\, X_{t}-B^{\frac{t}{T}}\right. \| \mathcal{F}_{\tau}\right]\right]}{\mathbb{P}\left(A_{T}\right)} \\
& =\frac{\mathbb{E}\left[\mathbf{1}_{A_{T}} \mathbf{1}_{\{\tau<(1-\delta) T\}} \mathbb{E}\left[\sup _{\tau \leq t<T} \left\lvert\, X_{t}-B^{\frac{t}{T}}\right. \| \mathcal{F}_{\tau}\right]\right]}{\mathbb{P}\left(A_{T}\right)} .
\end{aligned}
$$

Applying the strong Markov property and the freezing lemma, we have
$$
\mathbb{E}\left[\left.\sup _{\tau<t<T}\left|X_{t}-B^{\frac{t}{T}}\right| \right\rvert\, \mathcal{F}_{\tau}\right]=\left.\mathbb{E}\left[\sup _{0 \leq s<T-r}\left|R_{s}-B^{\frac{r+s}{T}}\right|\right]\right|_{r=\tau},
$$
where $R_{s}:=X_{\tau+s}$ is a geometric Brownian motion independent of $\mathcal{F}_{\tau}$. Hence
$$
\begin{aligned}
& \frac{\mathbb{E}\left[\mathbf{1}_{A_{T}} \mathbf{1}_{\{\tau<(1-\delta) T\}} \sup _{\tau \leq t<T}\left|X_{t}-B^{\frac{t}{T}}\right|\right]}{\mathbb{P}\left(A_{T}\right)} \\
& =\frac{\mathbb{E}\left[\left.\mathbf{1}_{A_{T}} \mathbf{1}_{\{\tau<(1-\delta) T\}} \mathbb{E}\left[\sup _{0 \leq s<T-r}\left|R_{s}-B^{\frac{r+s}{T}}\right|\right]\right|_{r=\tau}\right]}{\mathbb{P}\left(A_{T}\right)} \\
& \leq \frac{\mathbb{E}\left[\left.\mathbf{1}_{A_{T}} \mathbf{1}_{\{\tau<(1-\delta) T\}} \mathbb{E}\left[\sup _{0 \leq s<T-r}\left|R_{s}+B^{\frac{r+s}{T}}\right|\right]\right|_{r=\tau}\right]}{\mathbb{P}\left(A_{T}\right)}
\end{aligned}
$$

We note that
$$
\begin{aligned}
\left.\mathbb{E}\left[\sup _{0 \leq s<T-r}\left|R_{s}+B^{\frac{r+s}{T}}\right|\right]\right|_{r=\tau} & \leq \mathbb{E}\left[\sup _{0 \leq s \leq T}\left|R_{s}+B\right|\right] \\
& \leq \mathbb{E}\left[\sup _{0 \leq s \leq 1}\left|R_{s}+B\right|\right]
\end{aligned}
$$
for all small enough $T$. Since $\sup _{0 \leq s \leq 1} R_{s}$ is an $L^{1}$ random variable, we deduce that for all small enough $T$, $\left.\mathbb{E}\left[\sup _{0 \leq s<T-r}\left|R_{s}+B^{\frac{r+s}{T}}\right|\right]\right|_{r=\tau}$ is almost surely bounded above by some $C$ depending not on $T$ or $\tau$. Thus,
$$
\begin{aligned}
\frac{\mathbb{E}\left[\mathbf{1}_{A_{T}} \mathbf{1}_{\{\tau<(1-\delta) T\}} \sup _{\tau \leq t<T}\left|X_{t}-B^{\frac{t}{T}}\right|\right]}{\mathbb{P}\left(A_{T}\right)} & =O(1) \mathbb{E}\left[\mathbf{1}_{\{\tau<(1-\delta) T\}} \mid A_{T}\right] \\
& =O(1) \mathbb{P}[\tau<(1-\delta) T \mid \tau \leq T] \\
& \rightarrow 0
\end{aligned}
$$
as $T \rightarrow 0$ by Lemma 5.

On the other hand, we estimate
\begin{equation}\label{lol2}\begin{aligned}
& \mathbb{E}\left[\left.\mathbf{1}_{\{\tau \geq(1-\delta)\}} \sup _{\tau<t<T}\left|X_{t}-B^{\frac{t}{T}}\right| \right\rvert\, A_{T}\right] \\
& =\mathbb{E}\left[\left.\mathbf{1}_{\{\tau \geq(1-\delta) T\}} \sup _{\tau<t<T}\left|X_{t}-B^{\frac{t}{T}}\right| \right\rvert\, A_{T}\right] \\
& =\mathbb{E}\left[\left.\mathbf{1}_{\{\tau \geq(1-\delta) T\}} \sup _{\tau<t<T}\left|X_{\tau}+X_{T}-X_{\tau}-B^{\frac{t}{T}}\right| \right\rvert\, A_{T}\right] \\
& =\mathbb{E}\left[\left.\mathbf{1}_{\{\tau \geq(1-\delta) T\}} \sup _{\tau<t<T}\left|X_{T}-X_{\tau}+B-B^{\frac{t}{T}}\right| \right\rvert\, A_{T}\right] \\
& \leq \mathbb{E}\left[\mathbf{1}_{\{\tau \geq(1-\delta) T\}} \sup _{\tau<t<T}\left|X_{T}-X_{\tau}\right| \mid A_{T}\right]+\mathbb{E}\left[\left.\mathbf{1}_{\{\tau \geq(1-\delta) T\}} \sup _{\tau<t<T}\left|B-B^{\frac{t}{T}}\right| \right\rvert\, \mathcal{A}_{T}\right] . \\
& =\mathbb{E}\left[\mathbf{1}_{\{\tau \geq(1-\delta) T\}} \sup _{\tau<t<T}\left|X_{T}-X_{\tau}\right| \mid A_{T}\right]+\mathbb{E}\left[\left.\mathbf{1}_{\{\tau \geq(1-\delta) T\}}\left|B-B^{\frac{\tau}{T}}\right| \right\rvert\, A_{T}\right] . \\
& \leq \frac{\mathbb{E}\left[\mathbf{1}_{\{(1-\delta) T \leq \tau \leq T\}} \sup _{\tau<t<T}\left|X_{T}-X_{\tau}\right|\right]}{\mathbb{P}\left(A_{T}\right)}+\frac{\mathbb{E}\left[\mathbf{1}_{\{(1-\delta) T \leq \tau \leq T\}}\left|B-B^{1-\delta}\right|\right]}{\mathbb{P}\left(A_{T}\right)} \\
& =\frac{\mathbb{E}\left[\mathbf{1}_{\{(1-\delta) T \leq \tau \leq T\}} \sup _{\tau<t<T}\left|X_{T}-X_{\tau}\right|\right]}{\mathbb{P}\left(A_{T}\right)}+\left|B-B^{1-\delta}\right| \frac{\mathbb{E}\left[\mathbf{1}_{\{(1-\delta) T \leq \tau \leq T\}}\right]}{\mathbb{P}\left(A_{T}\right)} . \\
& \leq \frac{\mathbb{E}\left[\mathbf{1}_{\{(1-\delta) T \leq \tau \leq T\}} \sup _{\tau<t<T}\left|X_{T}-X_{\tau}\right|\right]}{\mathbb{P}\left(A_{T}\right)}+\left|B-B^{1-\delta}\right| .
\end{aligned}\end{equation}

To estimate the first term above, we write $R_{t}:=X_{\tau+t}$ and note that by the strong Markov property of SDEs, $R_{t}$ is a geometric Brownian motion independent of $\mathcal{F}_{\tau}$ with the same parameters $\mu, \sigma$ as $X$ and initial condition $R_{0}=B$. Noting also that $X_{\tau}=B$, the first term reads
$$
\begin{aligned}
& \frac{\mathbb{E}\left[\mathbf{1}_{\{(1-\delta) T \leq \tau \leq T\}} \mathbb{E}\left[\sup _{0 \leq t \leq T-\tau}\left|R_{t}-B\right|\right]\right]}{\mathbb{P}\left(A_{T}\right)} \\
& \leq \frac{\mathbb{E}\left[\mathbf{1}_{\{(1-\delta) T \leq \tau \leq T\}} \mathbb{E}\left[\sup _{0 \leq t \leq \delta T}\left|R_{t}-B\right|\right]\right.}{\mathbb{P}\left(A_{T}\right)} \\
& \leq \mathbb{E}\left[\sup _{0 \leq t \leq \delta T}\left|R_{t}-B\right|\right]
\end{aligned}
$$
which tends to 0 as $T \rightarrow 0$ by standard estimates on SDE (see, for example \cite{baldi}, Theorem 9.1). Thus we have, for any $\delta>0$,
$$
\lim _{T \rightarrow 0^{+}} \mathbb{E}\left[\left.\sup _{\tau<t<T}\left|X_{t}-B^{\frac{t}{T}}\right| \right\rvert\, \mathcal{A}_{T}\right] \leq\left|B-B^{1-\delta}\right|
$$
which tends to 0 as $\delta \rightarrow 0$. Thus sending $\delta$ to 0 , we obtain the desired claim (\ref{claim2}). This completes the proof of (\ref{zero}).

Now we prove the $O(\sqrt{T})$ convergence rate. From (\ref{lol}) and (\ref{claim1}, we have
$$
(1-C) \mathbb{E}_{\mathbb{P}_{T}}\left[\sup _{0<t<T}\left|X_{t}-B^{\frac{t}{T}}\right|\right] \leq \mathbb{E}_{\mathbb{Q}_{G-\left(\frac{\mu}{\sigma}-\frac{\sigma}{2}\right)T, T}}\left[\sup _{0<t<T}\left|X_{t}-B^{\frac{t}{T}}\right|\right]+\mathbb{E}\left[\left.\sup _{\tau<t<T}\left|X_{t}-B^{\frac{t}{T}}\right| \right\rvert\, \mathcal{A}_{T}\right]
$$
for some fixed $0<C<\frac{1}{2}$. By Lemma 4, the first term on the right hand side above is of order $O(\sqrt{T})+$ $\left|\left(\frac{\mu}{\sigma}-\frac{\sigma}{2}\right) T\right|=O(\sqrt{T})$. Hence to prove the proposition, it will suffice to show that
$$
\mathbb{E}\left[\left.\sup _{\tau<t<T}\left|X_{t}-B^{\frac{t}{T}}\right| \right\rvert\, \mathcal{A}_{T}\right]=O(\sqrt{T}) .
$$

To this end, we write
$$
\begin{aligned}
\mathbb{E}\left[\left.\sup _{\tau<t<T}\left|X_{t}-B^{\frac{t}{T}}\right| \right\rvert\, \mathcal{A}_{T}\right]= & \mathbb{E}\left[\left.\mathbf{1}_{\left\{\tau<\left(1-T^{1 / 2}\right) T\right\}} \sup _{\tau<t<T}\left|X_{t}-B^{\frac{t}{T}}\right| \right\rvert\, \mathcal{A}_{T}\right] \\
& +\mathbb{E}\left[\left.\mathbf{1}_{\tau \geq\left(1-T^{1 / 2}\right) T} \sup _{\tau \leq t<T}\left|X_{t}-B^{\frac{t}{T}}\right| \right\rvert\, \mathcal{A}_{T}\right]
\end{aligned}
$$

Similarly as to the estimate of the first term in (\ref{lol2}), we deduce that
$$
\mathbb{E}\left[\left.\mathbf{1}_{\left\{\tau<\left(1-T^{1 / 2}\right) T\right\}} \sup _{\tau<t<T}\left|X_{t}-B^{\frac{t}{T}}\right| \right\rvert\, \mathcal{A}_{T}\right]=O(1) \mathbb{P}\left[\tau<\left(1-T^{1 / 2}\right) T \mid \tau \leq T\right]
$$

The proof of Lemma 5 shows that
$$
\mathbb{P}\left[\tau<\left(1-T^{1 / 2}\right) T \mid \tau \leq T\right]=O\left(T^{-1 / 2} \exp \left(-\frac{C_{1}}{T^{1 / 2}}\right)\right)
$$
which is certainly of order $O(\sqrt{T})$. Hence it is left to show that
$$
\mathbb{E}\left[\left.\mathbf{1}_{\tau \geq\left(1-T^{1 / 2}\right) T} \sup _{\tau \leq t<T}\left|X_{t}-B^{\frac{t}{T}}\right| \right\rvert\, \mathcal{A}_{T}\right]=O(\sqrt{T}) .
$$

But we may estimate
$$
\mathbb{E}\left[\left.\mathbf{1}_{\tau \geq\left(1-T^{1 / 2}\right) T} \sup _{\tau \leq t<T}\left|X_{t}-B^{\frac{t}{T}}\right| \right\rvert\, \mathcal{A}_{T}\right] \leq \mathbb{E}\left[\sup _{0 \leq t \leq T^{3 / 2}}\left|R_{t}-B\right|\right]+\left|B-B^{1-T^{1 / 2}}\right|
$$
where again $R_{t}:=X_{t+\tau}$. The first term above is of order $O\left(T^{3 / 4}\right)$ by standard estimates on solutions to SDE (see \cite{baldi}, Theorem 9.1), and hence a fortiori of order $O(\sqrt{T})$. On the other hand, we have
$$
\begin{aligned}
\left|B-B^{1-T^{1 / 2}}\right| & =B^{1-T^{1 / 2}}\left(B^{T^{1 / 2}}-1\right) \\
& \leq B\left(e^{T^{1 / 2} \ln B}-1\right) \\
& =B\left(1+(\ln B) T^{1 / 2}+o\left((\ln B) T^{1 / 2}\right)-1\right) \\
& =O(\sqrt{T})
\end{aligned}
$$
where we have applied a Taylor expansion in the second to last equality. Combining the two estimates above gives
$$
\mathbb{E}\left[\left.\mathbf{1}_{\tau \geq\left(1-T^{1 / 2}\right) T} \sup _{\tau \leq t<T}\left|X_{t}-B^{\frac{t}{T}}\right| \right\rvert\, \mathcal{A}_{T}\right]=O(\sqrt{T})
$$
which concludes the proof.
\end{proof}
\section{Pricing of Short Maturity Options}

We now apply Theorem 1 to derive an asymptotic expression for the price of a short maturity barrier option with general payoff. Suppose $X$ as defined in Section 1 is taken to be the model of a stock price process with initial price $S > 0$. For convenience, we restate the defining $\operatorname{SDE}$ for $X$ :
$$
d X_{t}=\mu X_{t} d t+\sigma X_{t} d W_{t}, \quad X_{0}= S \, \mathrm{a} . \mathrm{s}
$$

Consider an out of the money up-and-in Asian option written on the stock price $X$ with barrier $B>S$, strike price $K>0$ and maturity time $T>0$. 

We consider a general Lipschitz continuous payoff $\Psi: C[0, T] \to \mathbb R$. Thus, for a given stock price path $X_t$, the option pays off $\Psi(\{X_t\}_{0 \leq t \leq T})$ if $\sup_{0 \leq t \leq T} X_t \geq B$, and pays off zero otherwise. This will be enough to cover the European, Asian, and lookback style payoffs, for which we will provide explicit expressions for the price.

We assume for simplicity that we are in a market with no interest rates. As is well known, the fair price $C(B, K, T)$ of the option is then given by
$$
C(B, K, T)=\mathbb{E}\left[\Psi(\{X_t\}_{0 \leq t \leq T}) \cdot \mathbf{1}_{\left\{\max _{0 \leq t \leq T} X_{t} \geq B\right\}}\right]
$$

The main theorem of this section is as follows.
\begin{theorem}[Asymptotics for short maturity barrier option]. The fair price $C(B, K, T)$ of the barrier option satisfies the following short time asymptotics as $T \rightarrow 0^{+}$:
$$
C(B, K, T)=P(B, T)\left[\Psi(\{S(\frac{B}{S})^{t/T}\}_{0 \leq t \leq T}) +O(\sqrt{T})\right]
$$
where $P(B, T):=\mathbb{P}\left(\max _{0 \leq t \leq T} X_{t} \geq B\right)$. The implied constant in the $O$ notation depends only on $\sigma$, $\mu$, and $B$.
\end{theorem}

\begin{remark} $P(B, T)$ may be explicitly computed as
$$
P(B, T)=1+\left (\frac{B}{S}\right )^{-1} \Phi\left(\frac{\frac{\sigma^{2} T}{2}-\ln \frac{B}{S}}{\sigma \sqrt{T}}\right)-\Phi\left(\frac{\ln \frac{B}{S} -\frac{\sigma^2 T}{2}}{\sigma \sqrt{T}}\right)
$$
using the probability density function of the running maximum of a geometric Brownian motion. Here $\Phi$ denotes the CDF of the standard normal distribution.
\end{remark}
\begin{proof}
Only for this section, we use the notation $\mathbb E_{\mathbb P_T}$ for the expectation under the probability measure $\mathbb P_T$ given by

$$\mathbb P_T (E) := \frac{\mathbb P(E \cap A_T)}{\mathbb P(A_T)},$$
where $A_T$ is the event
$\{\max X_T \geq \frac{B}{S}\}$.
This is the same convention as in Theorem 1, just with $\frac{B}{S}$ in place of $B$ in the event $A_T$.

Conditioning on $A_T$, we have
$$
\begin{aligned}
C(B, K, T) & =\mathbb{E}\left[\left(\frac{1}{T} \int_{0}^{T} X_{t} d t-K\right)_{+} \mathbf{1}_{\left\{\max _{0 \leq t \leq T} X_{t} \geq B\right\}}\right] \\
& =\mathbb{P}\left(\max _{0 \leq t \leq T} X_{t} \geq B\right) \mathbb{E}_{\mathbb{P}_{T}} \left [\Psi(\{X_t\}_{0 \leq t \leq T})\right]
\end{aligned}
$$

To estimate the last term above, we note that
$$
\begin{aligned}
& \mathbb{E}_{\mathbb{P}_{T}}  [\Psi(\{X_t\}_{0 \leq t \leq T})]-\Psi(\{S(\frac{B}{S})^{t/T}\}_{0 \leq t \leq T})\\
& \leq \mathbb{E}_{\mathbb{P}_{T}}\left[ \left |\Psi(X_t) -\Psi(\{S(\frac{B}{S})^{t/T}\}_{0 \leq t \leq T}) \right |\right] \\
& \leq L \mathbb{E}_{\mathbb{P}_{T}}\left[\sup_{0 \leq t \leq T} \left |X_t -(\frac{B}{S})^{t/T} \right |\right] 
\end{aligned}
$$
where $L$ denotes the Lipschitz constant of $\Phi$.  

We now claim that 
$$\mathbb{E}_{\mathbb{P}_{T}}\left[\sup_{0 \leq t \leq T} |X_t -(\frac{B}{S})^{t/T}|\right] = O(\sqrt T)$$
as $T \to 0$, which will complete the proof. Indeed, making the substitution $Y_t := \frac{X_t}{S}$, we find that $Y_t$ is a geometric Brownian motion satisfying the conditions of Theorem 1, thus 
$$\mathbb{E}_{\mathbb{P}_{T}}\left[\sup_{0 \leq t \leq T} |Y_t - (\frac{B}{S})^{t/T}|\right] = O(\sqrt T)$$
and so the conclusion for $X_t$ follows.
\end{proof}

We may now read off the short time asymptotics for the fair price of the Asian and lookback style payoffs from Theorem 6. We recall that these are given respectively by

$$\Psi_{Asian} (\{X_t\}_{0 \leq t \leq T}) = \left (\frac{1}{T} \int_0^T X_t dt - K \right )_+$$
$$\Psi_{lookback} (\{X_t\}_{0 \leq t \leq T}) = \left(\sup_{0 \leq t \leq T} X_t - K \right)+$$
\begin{proposition} The short time asymptotics for the fair price $C_{European}$, $C_{Asian}$, $C_{lookback}$ of the Asian and lookback style barrier options are given by

$$C_{Asian}(B, K, T)= \left [ 1+\left (\frac{B}{S}\right )^{-1} \Phi\left(\frac{\frac{\sigma^{2} T}{2}-\ln \frac{B}{S}}{\sigma \sqrt{T}}\right)-\Phi\left(\frac{\ln \frac{B}{S} -\frac{\sigma^2 T}{2}}{\sigma \sqrt{T}}\right) \right ] \cdot $$
$$\left [\left(\frac{1}{\ln (\frac{B}{S} )} \right ) \left (\frac{B}{S}  - 1 \right ) - K + O(\sqrt{T})\right] $$

$$C_{lookback}(B, K, T)=\left [ 1+\left (\frac{B}{S}\right )^{-1} \Phi\left(\frac{\frac{\sigma^{2} T}{2}-\ln \frac{B}{S}}{\sigma \sqrt{T}}\right)-\Phi\left(\frac{\ln \frac{B}{S} -\frac{\sigma^2 T}{2}}{\sigma \sqrt{T}}\right) \right ]\left[B - K + O(\sqrt{T})\right]
$$
\end{proposition}

\begin{proof} The expressions follow from substituting the specific expressions for $\Psi$ into the statement of Theorem 6.\end{proof}

\section{Further directions}
We conclude the paper with some natural directions for further research. As mentioned in the introduction, convergence of pure jump Levy processes conditional on taking a large maximum value has been studied extensively. Our present paper concerns the case of geometric Brownian motion, which is purely continuous. We are thus led to ask - what happens in the general case of Levy processes with both continuous and jump components? Do we still observe a single large jump, or does the continuous component dominate growth? Does the limiting process exhibit deterministic behaviour, random behaviour, or both?

We also propose the case of diffusions with general state dependent coefficients as a natural extension - can the results of our paper be extended to general diffusions, under suitable regularity and growth assumptions on the coefficients? This would lead also to extensions of our short maturity option pricing results to the case of general local volatility models.

Finally, we pose the question of obtaining finer detail on the conditional distribution of the paths. Our present work establishes strong convergence of order $\sqrt T$. In analogy with the classical central limit theorem, we may ask if the distribution converges, after a normalization of order $\sqrt T$ to a familiar distribution, such as a Brownian bridge.
\printbibliography
\end{document}